\documentclass[11pt]{article}
\usepackage{geometry}
\usepackage{color}
\usepackage{float}
\usepackage{textcomp}
\usepackage{url}
\usepackage{amsthm}
\usepackage{amsmath}
\usepackage{amssymb}%
\usepackage{mdframed}
\usepackage{multirow}
\usepackage{changes}
\usepackage{hyperref}
\usepackage{colortbl}

\usepackage{mathtools}
\usepackage{booktabs}    %pacote estético para tabelas
\usepackage{subfigure} 
\usepackage{caption} 
\usepackage{multirow}
\usepackage{multicol}    %para ajustar alinhamentos de colunas
\usepackage{array}       %para ajustar espaçamentos
\usepackage{graphicx}
\graphicspath{{figuras/}}
\usepackage{empheq,amsmath}

\usepackage[most]{tcolorbox}
\definecolor{tableheader}{HTML}{4F81BD}

\newcommand{\shallowddots}{%
  \rotatebox[origin=c]{12}{$\ddots$}%
}

\usepackage
[ %linesnumbered, 
lined, 
boxruled, %more options => boxed  algoruled  tworuled  plain
commentsnumbered
]
{algorithm2e}
\DecMargin{0.1cm} %diminui a margem esquerda do algorithm
\usepackage{cleveref}

\newtheorem{theorem}{Theorem}[section]
\newtheorem{lemma}[theorem]{Lemma}
\newtheorem{corollary}[theorem]{Corollary}

\newtheorem{remark}[theorem]{Remark}

\newcommand{\R}{\mathbb{R}}

\newcommand{\la}{{\langle}}
\newcommand{\ra}{{\rangle}}

\newcommand{\beq}{\begin{equation}}
\newcommand{\eeq}{\end{equation}}
\newcommand{\bi}{\begin{itemize}}
\newcommand{\ei}{\end{itemize}}
\newcommand{\ba}{\begin{array}}
\newcommand{\ea}{\end{array}}

\newlength{\eqAlgoAfter}
\newlength{\eqAlgoBefore}
\begin{document}

\title{\textbf{On the Complexity of Lower-Order Implementations of Higher-Order Methods}}

\author{
Nikita Doikov\thanks{Machine Learning and Optimization Laboratory (MLO), École Polytechnique Fédérale de
	Lausanne (EPFL), Lausanne, Switzerland (nikita.doikov@epfl.ch). This author was supported by the Swiss State Secretariat for Education, Research and Innovation (SERI)
	under contract number 22.00133.}
    \and
Geovani Nunes Grapiglia\thanks{Department of Mathematical Engineering, ICTEAM Institute, Université catholique de Louvain, B-1348 Louvain-la-Neuve, Belgium (geovani.grapiglia@uclouvain.be). 
		     This author was partially supported by FRS-FNRS, Belgium (Grant CDR J.0081.23).} 
}

\date{October 9, 2025}

\maketitle

\begin{abstract}
In this work, we propose a method for minimizing non-convex functions with Lipschitz continuous $p$th-order derivatives, starting from $p \geq 1$. The method, however, only requires derivative information up to order $(p-1)$, since the $p$th-order derivatives are approximated via finite differences. To ensure oracle efficiency, instead of computing finite-difference approximations at every iteration, we reuse each approximation for $m$ consecutive iterations before recomputing it, with $m \geq 1$ as a key parameter. As a result, we obtain an adaptive method of order $(p-1)$ that requires no more than $\mathcal{O}(\epsilon^{-\frac{p+1}{p}})$ iterations to find an $\epsilon$-approximate stationary point of the objective function and that, for the choice $m=(p-1)n + 1$, where $n$ is the problem dimension, takes no more than $\mathcal{O}(n^{1/p}\epsilon^{-\frac{p+1}{p}})$ oracle calls of order $(p-1)$.
This improves previously known bounds for tensor methods with finite-difference approximations
in terms of the problem dimension.
\\[0.2cm]
\end{abstract}

\section{Introduction}

%%%%%%%%%%%%%%%%%%%%%%%%%%%%%%%%%%%%%%%%%%%%%%%%%%
%%%%%%%%%%%%%%%%%%%%%%%%%%%%%%%%%%%%%%%%%%%%%%%%%%
\subsection{Motivation}
%%%%%%%%%%%%%%%%%%%%%%%%%%%%%%%%%%%%%%%%%%%%%%%%%%
%%%%%%%%%%%%%%%%%%%%%%%%%%%%%%%%%%%%%%%%%%%%%%%%%%

We consider the problem of unconstrained minimization of a smooth function $f:\mathbb{R}^n \to \mathbb{R}$, which can be non-convex. The choice of appropriate algorithms, as well as their performance guarantees, depends fundamentally on the smoothness properties of $f$ and on the extent to which its derivative information can be computed. When for a given $x \in \mathbb{R}^{n}$ one can evaluate $f(x)$ as well as all derivatives $\nabla^{i} f(x)$ for $i = 1, \ldots, q$ for some $q \leq p$, we say that a $q$th-order oracle for $f$ is available. Any algorithm that relies on such an oracle is referred to as a $q$th-order method.

In this context, the Gradient Method is a first-order method that, when applied to an objective function $f$ with the Lipschitz continuous gradient, requires at most $\mathcal{O}\big(\epsilon^{-2}\big)$ calls to the first-order oracle to find an $\epsilon$-approximate stationary point \cite{Nes}; that is, a point $\bar{x}$ satisfying $\|\nabla f(\bar{x})\| \le \epsilon$.

Turning to second-order methods, Newton's Method with Cubic Regularization \cite{griewank1981modification,NesPol,cartis2011adaptive1}, when applied to an objective function $f$ with an Lipschitz continuous Hessian, requires at most $\mathcal{O}(\epsilon^{-3/2})$ calls to the second-order oracle to find an $\epsilon$-approximate stationary point~\cite{NesPol}. More generally, if $f$ has a Lipschitz continuous $p$th-order derivative, the methods based on regularized $p$th-order Taylor models \cite{Birgin} require at most $\mathcal{O}(\epsilon^{-(p+1)/p})$ calls to the $p$th-order oracle to find an $\epsilon$-approximate stationary point. As shown in \cite{Carmon}, this represents an optimal worst-case complexity bound for $p$th-order methods applied to the class of functions with Lipschitz continuous $p$th-order derivatives.

In minimizing a $p$-times differentiable function $f$, one often has access only to a $q$-th order oracle, where $q < p$. For example, when $f(x)$ can be evaluated only through a black-box simulation or experiment, the gradient vector $\nabla f(x)$ is not readily available. In such cases, one must rely solely on function evaluations to minimize $f$, which motivates the use of derivative-free (zeroth-order) me\-thods \cite{CSV,AudetHare,Larson}. Another example arises when calibrating the parameters of ODE or PDE models to fit a given dataset. Although the error function may be twice differentiable, computing the Hessian via adjoint equations can be computationally prohibitive or practically challenging \cite{Stapor}, forcing users to rely solely on a first-order oracle \cite{Kapfer,Plessix,Liu}. 

A common approach to exploiting $p$th-order smoothness with a lower-order oracle is to approximate the unavailable higher-order derivatives 
via \textit{finite differences} of the available lower-order information. For example, by approximating gradient vectors with forward finite differences of function values, one can design zeroth-order implementations of first-order methods \cite{grapiglia1} that find $\epsilon$-approximate stationary points using no more than $\mathcal{O}\left(n\epsilon^{-2}\right)$ calls to the zeroth-order oracle (i.e., function evaluations), where $n$ is the problem dimension. Similarly, finite differences of gradients can be used to approximate Hessian matrices in second-order methods, leading to first-order implementations \cite{CGT1,GGS} with worst-case oracle complexity $\mathcal{O}\left(n\epsilon^{-3/2}\right)$. Recently, in \cite{DoikovGrapiglia}, we established an improved complexity bound of $\mathcal{O}\left(n^{1/2}\epsilon^{-3/2}\right)$ for a \textit{lazy} variant of Newton’s Method with Cubic Regularization~\cite{doikov2023second}, in which finite-difference Hessian approximations are computed only once every $n$ iterations and kept fixed in between. Further developing the lazy technique, \cite{CartisJerad} proposed a method of order $(p-1)$ in which finite-difference approximations of the $p$th-order derivatives are computed only once every $m$ iterations and updated in between via quasi-tensor rules \cite{Welzel}, where $m \in \mathbb{N}\setminus\{0\}$ is a user-defined parameter. It was shown in \cite{CartisJerad} that this method requires at most 
$\mathcal{O}\Bigl(n \max\!\bigl\{\epsilon_{1}^{-\frac{p+1}{p}}, \, \epsilon_{2}^{-\frac{p+1}{p-1}}\bigr\}\Bigr)$ calls to the oracle of order $(p-1)$ to find an $(\epsilon_{1},\epsilon_{2})$-approximate second-order stationary point, i.e., an $\epsilon_{1}$-approximate stationary point of $f$ at which the smallest eigenvalue of the Hessian is at least~$-\epsilon_{2}$.

%%%%%%%%%%%%%%%%%%%%%%%%%%%%%%%%%%%%%%%%%%%%%%%%%%
%%%%%%%%%%%%%%%%%%%%%%%%%%%%%%%%%%%%%%%%%%%%%%%%%%
\subsection{Contributions}
%%%%%%%%%%%%%%%%%%%%%%%%%%%%%%%%%%%%%%%%%%%%%%%%%%
%%%%%%%%%%%%%%%%%%%%%%%%%%%%%%%%%%%%%%%%%%%%%%%%%%

In this work, we propose a new method of order $(p-1)$ for minimizing $p$-times differentiable functions with Lipschitz continuous $p$th-order derivatives. To ensure oracle efficiency, we extend the lazy approach from \cite{DoikovGrapiglia} to the general case $p \ge 1$, where each finite-difference approximation of the $p$th-order derivative is reused for at most $m$ consecutive iterations before a new approximation is computed. We show that, by choosing $m = (p-1)n + 1$, our method requires at most 
$$
\ba{c}
\mathcal{O}\Big(n^{1/p} \epsilon^{-\frac{p+1}{p}}\Big)
\ea
$$ 
calls to the oracle of order $(p-1)$ to find an $\epsilon$-approximate stationary point. With respect to explicit dependence on the problem dimension $n$, this is, to our knowledge, the sharpest known oracle complexity bound for $(p-1)$th-order methods applied to functions with Lipschitz-continuous $p$th derivatives.
Importantly, our method is adaptive, automatically adjusting the Lipschitz constant of the $p$th derivative and the parameter of the finite-difference approximation, without the need to fix them in advance.

%%%%%%%%%%%%%%%%%%%%%%%%%%%%%%%%%%%%%%%%%%%%%%%%%%
%%%%%%%%%%%%%%%%%%%%%%%%%%%%%%%%%%%%%%%%%%%%%%%%%%
\subsection{Contents}
%%%%%%%%%%%%%%%%%%%%%%%%%%%%%%%%%%%%%%%%%%%%%%%%%%
%%%%%%%%%%%%%%%%%%%%%%%%%%%%%%%%%%%%%%%%%%%%%%%%%%

The paper is organized as follows. 
Section~\ref{SectionProblem} introduces the problem and presents the necessary au\-xiliary results. 
In Section~\ref{SectionMethod}, we describe our new method and analyze its worst-case oracle complexity. 
Section~\ref{SectionDiscussion} concludes with a discussion of open problems and directions for future research.

%%%%%%%%%%%%%%%%%%%%%%%%%%%%%%%%%%%%%%%%%%%%%%%%%%
%%%%%%%%%%%%%%%%%%%%%%%%%%%%%%%%%%%%%%%%%%%%%%%%%%
\subsection{Notation}
\label{SectionNotation}
%%%%%%%%%%%%%%%%%%%%%%%%%%%%%%%%%%%%%%%%%%%%%%%%%%
%%%%%%%%%%%%%%%%%%%%%%%%%%%%%%%%%%%%%%%%%%%%%%%%%%

For a differentiable function $f: \R^n \to \R$,  we denote by $\nabla^q f(x)$
the derivative of order $q \geq 1$,
which is a $q$-linear symmetric form.
The value of this form on a set of fixed directions $h_1, \ldots, h_q \in \R^n$
is denoted by
$
\nabla^q f(x)[h_1, \ldots, h_q]  \in  \R,
$
which is the $q$th-order directional derivative of $f$ along the given directions.
When all directions are the same, $h_1 \equiv \ldots \equiv h_q \equiv h \in \R^n$
we use the shorthand $\nabla^q f(x)[h]^q \in \R$.
More generally, for an arbitrary $1 \leq \ell \leq q$, we use the convenient notation $\nabla^q f(x)[h]^{\ell}$
for the $(q - \ell)$-form with the first $\ell$ directions substituted as $h$; by definition, it satisfies
$$
\ba{rcl}
\nabla^q f(x)[h]^{\ell} [u_1, \ldots, u_{q - \ell}]  
&\equiv& \nabla^q f(x)[h, \ldots h, u_1, \ldots, u_{q - \ell}] 
\;\; \in \;\; \R,
\quad \text{for all }
u_1, \ldots, u_{q - \ell} \in \R^n.
\ea
$$

We fix the standard Euclidean inner product in our space,
$\la x, y \ra := \sum_{i = 1}^n x^{(i)} y^{(i)}$, for any $x, y \in \R^n$,
and denote by $e_1, \ldots, e_n \in \R^n$ the canonical basis.
Using the inner product, we treat the $1$-form $\nabla f(x) \in \R^n$ as the gradient vector:
$
\la \nabla f(x), u \ra \equiv \nabla f(x)[u]  \in \R$
for  $x, u \in \R^n$. And for general $q \geq 1$ and a fixed direction $h \in \R^n$,
we can treat $\nabla^q f(x)[h]^{q - 1} \in \R^n$ as the vector, which satisfies
$\la \nabla^q f(x)[h]^{q - 1}, u \ra \equiv \nabla^q f(x)[h]^{q - 1}[u] \in \R$, for all $x, h, u \in \R^n$.

We denote by $\| \cdot \|$ the standard Euclidean norm for vectors, $\| x \| := \la x, x \ra^{1/2}$, $x \in \R^n$.
Correspondingly, for an arbitrary $q$-linear (not necessary symmetric) function $T[h_1, \ldots, h_q] \in \R$, $q \geq 1$ we use the induced operator norm:
$$
\ba{rcl}
\|T\| & := & \max\limits_{\substack{h_1, \ldots, h_q \in \R^n : \\ \| h_i \| \leq 1,  \forall 1 \leq i \leq q}}
| T[h_1, \ldots, h_q] |.
\ea
$$
When $T$ is symmetric, it holds (see, e.g., Appendix~1 in \cite{nesterov1994interior}) $\|T\| = \max_{h \in \R^n : \| h \| \leq 1} |T[h]^q|$.

In what follows, we will use the following construction. For a given symmetric $(q - 1)$-linear form~$E$,
and an arbitrary vector $e \in \R^n$, we denote by $T = E \otimes e$
the $q$-linear form, defined by
$$
\ba{rcl}
T[h_1, \ldots h_q] & \equiv & E[h_1, \ldots, h_{q - 1}] \cdot \la e, h_q \ra \;\; \in \;\; \R,
\qquad
h_1, \ldots, h_q \in \R^n.
\ea
$$
It holds that $\|T\| = \|E \| \cdot \| e\|$. 
Clearly, $T$ is not symmetric in general. 
Introducing the function $g(h) = \frac{1}{q} T[h]^q \in \R$, $h \in \R^{n}$, we find that
its gradient vector is equal to
\beq \label{EeGrad}
\ba{rcl}
\nabla g(h) & = & \frac{1}{q}\Bigl( 
E[h]^{q - 1} e + (q - 1) \la e, h \ra E[h]^{q - 2}
\Bigr)
\;\; \in \;\; \R^n.
\ea
\eeq
It is possible to make tensor $T$ symmetric, using the following standard symmetrization operation:
$$
\ba{rcl}
P_{sym}(T)[h_1, \ldots, h_q]
& := & \frac{1}{q!} \sum\limits_{\sigma \in S_q} T[ h_{\sigma(1)}, \ldots, h_{\sigma(q)} ],
\ea
$$
for which we have: $P_{sym}(T)[h]^q \equiv T[h]^q$, $\nabla_h P_{sym}(T)[h]^q \equiv q P_{sym}(T)[h]^{q - 1}$,  
and $\| P_{sym}(T)\| \leq \| T \|$.

Throughout our analysis, we will use Young's inequality:
\begin{equation} \label{Young}
	\ba{rcl}
	xy & \leq & \frac{x^{\alpha}}{\alpha} + \frac{y^{\beta}}{\beta} 
	\qquad \text{for } 
	x, y \ge 0 \text{ and } \frac{1}{\alpha} + \frac{1}{\beta} = 1, \, \alpha,\beta > 1.
	\ea
\end{equation}

%%%%%%%%%%%%%%%%%%%%%%%%%%%%%%%%%%%%%%%%%%%%%%%%%%
%%%%%%%%%%%%%%%%%%%%%%%%%%%%%%%%%%%%%%%%%%%%%%%%%%
\section{Problem Formulation and Auxiliary Results}
\label{SectionProblem}
%%%%%%%%%%%%%%%%%%%%%%%%%%%%%%%%%%%%%%%%%%%%%%%%%%
%%%%%%%%%%%%%%%%%%%%%%%%%%%%%%%%%%%%%%%%%%%%%%%%%%

In this work we consider the smooth unconstrained optimization problem
$$
\ba{c}
\min\limits_{x \in \mathbb{R}^{n}}\,f(x),
\ea
$$
under the following assumptions:
\begin{mdframed}
\begin{itemize}
\item[\textbf{A1.}] $f:\mathbb{R}^{n}\to\mathbb{R}$ is a $p$-times differentiable function with $L$-Lipschitz continuous $p$th-order derivative ($p \geq 1$).
Thus,
\beq \label{LipTensor}
\ba{rcl}
\| \nabla^p f(x) - \nabla^p f(y) \| & \leq & L \|x - y\|, \qquad \forall x, y \in \R^n.
\ea
\eeq
\item[\textbf{A2.}] Function $f$ is bounded from below:
$f(x)\geq f_{low}$, for all $x\in\mathbb{R}^{n}$.
\end{itemize}
\end{mdframed}
It follows from A1 that
\begin{equation} \label{UpperTaylor}
\ba{rcl}
f(y) 
& \leq & 
\Phi_{x,p}(y) + 
\frac{L}{(p+1)!}\|y-x\|^{p+1},
\qquad\forall x,y \in\mathbb{R}^{n},
\ea
\end{equation}
where
$$
\ba{rcl}
\Phi_{x,p}(y)
& \equiv &
f(x) + \sum\limits_{i=1}^{p}\frac{1}{i!}\nabla^{i}f(x)[y-x]^{i}
\ea
$$
defines the $p$th-order Taylor polynomial of $f(\,\cdot\,)$ around $x$. Moreover, we also have
the global bound for the gradient approximation:
\begin{equation} \label{GradApprox}
\ba{rcl}
\|\nabla f(y)-\nabla \Phi_{x,p}(y)\|
& \leq  &
\frac{L}{p!}\|y-x\|^{p},
\qquad\forall x,y\in\mathbb{R}^{n}.
\ea
\end{equation}

Given a point $\bar{x}\in\mathbb{R}^{n}$
and a regularization parameter $\sigma > 0$, we define the following augmented Taylor approximation
\beq \label{OmegaDef}
\ba{rcl}
\Omega_{\bar{x}, \sigma, p}(y)
& := & 
f(\bar{x}) 
+ \sum\limits_{i=1}^{p-1}\frac{1}{i!}\nabla^{i}f(\bar{x})[y-\bar{x}]^{i}
+ \frac{1}{p!} \nabla^p f(\bar{x})[y - \bar{x}]^p
+ \frac{\sigma}{(p + 1)!} \|y - \bar{x}\|^{p + 1}.
\ea
\eeq
According to~\eqref{UpperTaylor}, this is a global \textit{upper approximation}
of our objective, i.e., $f(y) \leq \Omega_{\bar{x}, \sigma, p}(y)$, when the regularization parameter is large enough: $\sigma \geq L$.
However, in general, this approximation might not be convex.
 It was shown in~\cite{nesterov2021implementable}
 that for convex function $f(\cdot)$,
 the augmented Taylor model~\eqref{OmegaDef} is convex for $\sigma \geq pL$, and, therefore, its minimizer can be efficiently computed.

Let us consider a $p$-linear tensor $T$ that approximates $p$th-order derivative of the objective, $T \approx \nabla^p f(\bar{x})$,
and use it in the following model:
\begin{equation} \label{ModelDef}
\ba{rcl}
	M_{\bar{x},\sigma,p}(y)
	& := & 
	f(\bar{x})
	+ \sum\limits_{i=1}^{p-1}\frac{1}{i!}\nabla^{i}f(\bar{x})[y-\bar{x}]^{i}
	+ \frac{1}{p!}T[y - \bar{x}]^{p}
	+ \frac{\sigma}{(p+1)!}\|y - \bar{x}\|^{p+1}.
\ea
\end{equation}
Therefore, we keep all derivatives of our function up to order $(p - 1)$ in our model, while approximating the $p$th derivative.
As we will see further, such an approximation can be done efficiently using lower derivatives of the function.
Note that model~\eqref{ModelDef} depends only on the symmetric part $P_{sym}(T)$ of our tensor,
while we do not assume that $T$ is symmetric. The gradient of our model is given by
$$
\ba{rcl}
\nabla M_{\bar{x}, \sigma, p}(y) & = & 
\sum\limits_{i = 1}^{p - 1} \frac{1}{(i - 1)!} \nabla^i f(\bar{x})[y - \bar{x}]^{i - 1} + \frac{1}{(p - 1)!} P_{sym}(T)[y - \bar{x}]^{p - 1}
+ \frac{\sigma}{p!}\|y - \bar{x}\|^{p - 1}(y - \bar{x}),
\ea
$$
and, correspondingly, the difference between $\nabla M_{\bar{x}, \sigma, p}(\cdot)$ and $\nabla \Omega_{\bar{x}, \sigma, p}( \cdot )$ 
can be bounded as
\beq \label{GradDiffBound}
\ba{rcl}
& & \!\!\!\!\!\!\!\!\!\!\!\!\!\!\!\!\!
\|\nabla M_{\bar{x}, \sigma, p}(y) -  \nabla \Omega_{\bar{x}, \sigma, p}(y) \|
\;\; = \;\;
\frac{1}{(p - 1)!} \| (P_{sym}(T) -  \nabla^p f(\bar{x}) )[y - \bar{x}]^{p - 1} \| \\
\\
& \leq & 
\frac{1}{(p - 1)!} \| P_{sym}(T) -  \nabla^p f(\bar{x}) \| \cdot \|y - \bar{x} \|^{p - 1}
\;\; \leq \;\;
\frac{1}{(p - 1)!} \|T - \nabla^p f(\bar{x}) \| \cdot \|y - \bar{x} \|^{p - 1}.
\ea
\eeq

The following lemma is one of the main tools of our analysis that relates the length of the step $\|x^+ - \bar{x}\|$ of a method,
with the gradient norm $\| \nabla f(x^+)\|$ at new point, where $x^+$ is an inexact minimizer to our approximate model~\eqref{ModelDef}.

\begin{lemma}
\label{LemmaGradR}
Suppose that Assumption A1 holds and let $x^{+}$ be an inexact minimizer of $M_{\bar{x},\sigma,p}(\,\cdot\,)$, defined in (\ref{ModelDef}), satisfying the following condition 
\beq \label{InexCond}
\ba{rcl}
\|\nabla M_{\bar{x},\sigma,p}(x^{+})\|
& \leq &
\frac{\sigma}{2p!}\|x^{+}-\bar{x}\|^{p}.
\ea
\eeq
If $\sigma\geq 2L$ and, for some $z\in\mathbb{R}^{n}$ and $\delta>0$, we have
\beq \label{TCond}
\ba{rcl}
\|T-\nabla^{p}f(z)\|
&\leq&
\delta,
\ea
\eeq
then
\beq \label{GradUpBound}
\ba{rcl}
\|\nabla f(x^{+})\|^{\frac{p+1}{p}}
& \leq &
3^{1/p}\sigma^{1/p} \Bigl[
4\sigma \|x^{+}-\bar{x}\|^{p+1} 
+ \frac{\delta^{p+1}}{p[(p-1)!]^{p+1}\sigma^{p}} 
+ \frac{L^{p+1}\|z-\bar{x}\|^{p+1}}{p[(p-1)!]^{p+1}\sigma^{p}}
\Bigr].
\ea
\eeq
\end{lemma}
\begin{proof}
Let us denote $r = \|x^{+}-\bar{x}\|$. From the previous reasoning, triangle inequality, and the equation $\nabla \Omega_{\bar{x}, \sigma, p}(x^+) = \nabla \Phi_{\bar{x}, \sigma, p}(x^+) + \frac{\sigma}{p!}r^{p - 1} (x^+ - x)$,
%(\ref{OmegaDef}), (\ref{ModelDef}), (\ref{GradApprox}), (\ref{InexCond}), 
it follows that
\begin{equation} \label{GradReasoning}
\ba{rcl}
 & & \!\!\!\!\!\!\!\!\!\!\!\!\!\!\!\!\!\!\!\!\!\!\!\! 
 \|\nabla f(x^{+})\| \\
 \\
 &\leq& \|\nabla f(x^{+})-\nabla\Omega_{\bar{x},\sigma,p}(x^{+})\|+\|\nabla\Omega_{\bar{x},\sigma,p}(x^{+})-\nabla M_{\bar{x},\sigma,p}(x^{+})\|+\|\nabla M_{\bar{x},\sigma,p}(x^{+})\| \\
\\
&\overset{(\ref{GradDiffBound}),(\ref{InexCond})}{\leq}&
\|\nabla f(x^{+}) - \nabla \Phi_{x,p}(x^{+})\|
+ \frac{\sigma}{p!}r^{p}
+ \frac{1}{(p-1)!} \bigl\| T-\nabla^{p}f(\bar{x})) \bigr\| r^{p - 1}
+ \frac{\sigma}{2p!}r^{p} \\
\\
&\overset{(\ref{GradApprox})}{\leq}& \bigl(\frac{L}{p!}+\frac{\sigma}{p!} + \frac{\sigma}{2p!}\bigr)r^{p}
+ \frac{1}{(p-1)!}\|T-\nabla^{p}f(\bar{x})\|r^{p-1} \\
\\
&\leq& 
\bigl(\frac{L}{p!}+\frac{\sigma}{p!}+\frac{\sigma}{2p!}\bigr) r^{p}
+ \frac{1}{(p-1)!} \bigl(\|T-\nabla^{p}f(z)\|
+ \|\nabla^{p}f(z)-\nabla^{p}f(\bar{x})\| \bigr)r^{p-1} \\
\\
&\overset{(\ref{TCond}),(\ref{LipTensor})}{\leq}& 
\frac{2\sigma}{p!}r^{p}
+ \Bigl(\frac{\delta}{(p-1)!}+\frac{L}{(p-1)!} \|z - \bar{x}\| \Bigr)r^{p-1},
\ea
\end{equation}
where in the last inequality we also used that $\sigma\geq 2L$. In view of (\ref{GradReasoning}) and using the inequality $(a+b+c)^{q}\leq 3^{q-1}(a^{q}+b^{q}+c^{q})$ for $a,b,c\geq 0$ and $q := \frac{p + 1}{p} \geq 1$, we get
\small
\beq \label{GradReasoning2}
\ba{rcl}
\|\nabla f(x^{+})\|^{\frac{p+1}{p}}
&\leq&
\Bigl[ \frac{2\sigma}{p!}r^{p}+\frac{\delta}{(p-1)!}r^{p-1}+\frac{L}{(p-1)!}\|z-\bar{x}\|r^{p-1}\Bigr]^{\frac{p+1}{p}} \\
\\
&\leq & 
3^{1/p}\Bigl[ \bigl(\frac{2}{p!}\bigr)^{\frac{p+1}{p}} \sigma^{\frac{p+1}{p}}r^{p+1} + \frac{\delta^{\frac{p+1}{p}}}{[(p-1)!]^{\frac{p+1}{p}}} r^{\frac{(p-1)(p+1)}{p}}
+ \frac{L^{\frac{p+1}{p}}\|z-\bar{x}\|^{\frac{p+1}{p}}}{[(p-1)!]^{\frac{p+1}{p}}}r^{\frac{(p-1)(p+1)}{p}} \Bigr] \\
\\
&=&
3^{1/p}\sigma^{1/p} \Bigl[ \bigl(\frac{2}{p!}\bigr)^{\frac{p+1}{p}}\sigma r^{p+1}
+ \frac{\delta^{\frac{p+1}{p}}}{[(p-1)!]^{\frac{p+1}{p}}\sigma^{1/p}}r^{\frac{(p-1)(p+1)}{p}} 
+ \frac{L^{\frac{p+1}{p}}\|z-\bar{x}\|^{\frac{p+1}{p}}}{[(p-1)!]^{\frac{p+1}{p}}\sigma^{1/p}}r^{\frac{(p-1)(p+1)}{p}} \Bigr].
\ea
\eeq
\normalsize
When $p = 1$, this inequality immediately gives~(\ref{GradUpBound}). Now, we assume that $p \geq 2$.

Using Young's inequality \eqref{Young} with $\alpha = p$ and $\beta = \frac{p}{p-1} > 1$, the second term inside the brackets in (\ref{GradReasoning2}) can be bounded as follows
\beq \label{GradReasoning3}
\ba{rcl}
    \frac{\delta^{\frac{p+1}{p}}}{[(p - 1)!]^{\frac{p+1}{p}}\sigma^{1/p}}r^{\frac{(p-1)(p+1)}{p}}
    &=&
    \left( \frac{\delta^{\frac{p+1}{p}}}{[(p-1)!]^{\frac{p+1}{p}}\sigma} \right) \left(\sigma^{\frac{p-1}{p}}r^{\frac{(p-1)(p+1)}{p}} \right) \\
    \\
    &\leq & 
    \frac{1}{p} \left( \frac{\delta^{\frac{p+1}{p}}}{[(p-1)!]^{\frac{p+1}{p}}\sigma} \right)^{p}  
    + \frac{p - 1}{p}  \left( \sigma^{\frac{p-1}{p}}r^{\frac{(p-1)(p+1)}{p}} \right)^{\frac{p}{p-1}}  \\
    \\
    &=&
    \frac{\delta^{p+1}}{p[(p-1)!]^{p+1}\sigma^{p}} + \frac{(p-1)}{p}\sigma r^{p+1} \\
    \\
    & < &
    \frac{\delta^{p+1}}{p[(p-1)!]^{p+1}\sigma^{p}} + \sigma r^{p+1}.
\ea
\eeq
Using again Young's inequality, we can also bound the third term inside the brackets in (\ref{GradReasoning2}),
\beq \label{GradReasoning4}
\ba{rcl}
    \frac{L^{\frac{p+1}{p}}\|z-\bar{x}\|^{\frac{p+1}{p}}}{[(p-1)!]^{\frac{p+1}{p}}\sigma^{1/p}}r^{\frac{(p-1)(p+1)}{p}}
    &=&
    \left(\frac{L^{\frac{p+1}{p}}\|z-\bar{x}\|^{\frac{p+1}{p}}}{[(p-1)!]^{\frac{p+1}{p}}\sigma}\right)
    \left(\sigma^{\frac{p-1}{p}}r^{\frac{(p-1)(p+1)}{p}}\right) \\
    \\
    &\leq&
    \frac{1}{p} \left(\frac{L^{\frac{p+1}{p}}\|z-\bar{x}\|^{\frac{p+1}{p}}}{[(p-1)!]^{\frac{p+1}{p}}\sigma}\right)^{p}
    + \frac{p - 1}{p} \left(\sigma^{\frac{p-1}{p}}r^{\frac{(p-1)(p+1)}{p}}\right)^{\frac{p}{p-1}} \\
    \\
    &=&
    \frac{L^{p+1}\|z-\bar{x}\|^{p+1}}{p[(p-1)!]^{p+1}\sigma^{p}} + \frac{(p-1)\sigma r^{p+1}}{p} \\
    \\
    &<& 
    \frac{L^{p+1}\|z-\bar{x}\|^{p+1}}{p[(p-1)!]^{p+1}\sigma^{p}} + \sigma r^{p+1}.
\ea
\eeq
Finally, combining (\ref{GradReasoning2}), (\ref{GradReasoning3}) and (\ref{GradReasoning4}), and using the inequality 
$\bigl(\frac{2}{p!}\bigr)^{\frac{p+1}{p}}\leq 1$, $p \geq 2$, we conclude that
$$
\ba{rcl}
    \|\nabla f(x^{+})\|^{\frac{p+1}{p}}
    & \leq &
    3^{1/p}\sigma^{1/p}\left[ 3\sigma r^{p+1} 
    	+ \frac{\delta^{p+1}}{p[(p-1)!]^{p+1}\sigma^{p}}
    	+ \frac{L^{p+1}\|z-\bar{x}\|^{p+1}}{p[(p-1)!]^{p+1}\sigma^{p}}\right] \\
        \\
    & \leq & 
        3^{1/p}\sigma^{1/p}\left[ 4\sigma r^{p+1} 
    	+ \frac{\delta^{p+1}}{p[(p-1)!]^{p+1}\sigma^{p}}
    	+ \frac{L^{p+1}\|z-\bar{x}\|^{p+1}}{p[(p-1)!]^{p+1}\sigma^{p}}\right].
\ea
$$
That is, (\ref{GradUpBound}) is true for all $p \geq 1$.
\end{proof}

The next lemma shows that minimizing an approximate upper model of the objective leads to a progress in terms of the function value.

\begin{lemma} \label{LemmaFuncProg}
Suppose that Assumption A1 holds and let $x^{+}$ be an inexact minimizer of $M_{\bar{x},\sigma,p}(\,\cdot\,)$, defined in (\ref{ModelDef}), satisfying the following condition
\begin{equation} \label{MonCond}
\ba{rcl}
M_{\bar{x},\sigma,p}(x^{+}) & \leq & f(\bar{x}).
\ea
\end{equation}
If $\sigma\geq 2L$ and, for some $z\in\mathbb{R}^{n}$ and $\delta>0$, (\ref{TCond}) holds, then
\begin{equation} \label{FuncProg}
\ba{rcl}
f(\bar{x})-f(x^{+})
& \geq & 
\frac{\sigma}{4\cdot (p+1)!}\|x^{+}-\bar{x}\|^{p+1}
- \frac{(8(p + 1))^p \cdot\left( \delta^{p+1}+L^{p+1}\|z-\bar{x}\|^{p+1} \right)}{\sigma^{p}\cdot (p+1)!}.
\ea
\end{equation}
\end{lemma}
\begin{proof}
Let $r=\|x^{+} - \bar{x} \|$. We have
\beq \label{FuncReasoning}
\ba{rcl}
    f(x^{+})
    &\overset{(\ref{UpperTaylor})}{\leq}&
    \Omega_{\bar{x},L,p}(x^{+}) 
    \;\; = \;\;
    \Omega_{\bar{x},\sigma,p}(x^{+}) + \frac{L-\sigma}{(p+1)!}r^{p+1} \\
    \\
    &=& 
    M_{\bar{x},\sigma,p}(x^{+})+\frac{1}{p!} \left(\nabla^{p}f(\bar{x}) - T\right)[x^{+} - \bar{x}]^{p}
    + \frac{L-\sigma}{(p+1)!}r^{p+1} \\
    \\
    &\leq& 
    f(\bar{x}) + \frac{1}{p!}\|\nabla^{p}f(\bar{x}) - T\|r^{p}+\frac{L-\sigma}{(p+1)!}r^{p+1} \\
    \\
    &\leq& 
    f(\bar{x}) + \frac{1}{p!}\left(\|\nabla^{p}f(\bar{x}) - \nabla^{p}f(z)\|+\|\nabla^{p}f(z)-T\|\right)r^{p} + \frac{L-\sigma}{(p+1)!}r^{p+1} \\
    \\
    &\overset{(\ref{LipTensor}), (\ref{TCond})}{\leq}& 
    f(\bar{x}) + \frac{1}{p!}\left(L\|z-\bar{x}\| + \delta\right)r^{p} - \frac{\sigma}{2\cdot (p+1)!}r^{p+1},
\ea
\eeq
where in the last inequality we used the assumption $\sigma\geq 2L$. Using Young's inequality~\eqref{Young} with $\alpha = p+1$ and $\beta = \frac{p+1}{p}$, and the inequality $(a+b)^{q}\leq 2^{q-1}(a^{q}+b^{q})$ for $a,b\geq 0$ and $q := p + 1$, we get
\beq \label{FuncReasoning2}
\ba{rcl}
    \frac{1}{p!}\left(L\|z-\bar{x}\| + \delta\right)r^{p}
    &=&
    \left(\frac{4^{\frac{p}{p+1}}[(p + 1)!]^{\frac{p}{p+1}}\left(\delta+L\|z-\bar{x}\|\right)}{\sigma^{\frac{p}{p+1}}\cdot p!}\right) 
    \left(\frac{\sigma^{\frac{p}{p+1}}r^{p}}{4^{\frac{p}{p+1}}[(p+1)!]^{\frac{p}{p+1}}}\right) \\
    \\
    &\leq&
    \frac{1}{p + 1} \left(\frac{4^{\frac{p}{p+1}}[(p + 1)!]^{\frac{p}{p+1}}\left(\delta+L\|z-\bar{x}\|\right)}{\sigma^{\frac{p}{p+1}}\cdot p!}\right)^{p+1}
    + \frac{p}{p + 1} \left(\frac{\sigma^{\frac{p}{p+1}}r^{p}}{4^{\frac{p}{p+1}}[(p+1)!]^{\frac{p}{p+1}}}\right)^{\frac{p+1}{p}} \\
    \\
    &=& 
    \frac{4^{p}[(p + 1)!]^{p}\left(\delta+L\|z-\bar{x}\|\right)^{p+1}}{(p+1)\cdot [p!]^{p+1}\cdot \sigma^{p}}
    + \frac{p\cdot\sigma\cdot r^{p+1}}{4(p+1)\cdot (p+1)!} \\
    \\
    &\leq& 
    \frac{(4(p + 1))^{p} \left(\delta+L\|z-\bar{x}\|\right)^{p+1}}{\sigma^{p}\cdot (p+1)!}
    + \frac{\sigma}{4\cdot (p+1)!}r^{p+1} \\
    \\
    &\leq& 
    \frac{(8(p + 1))^p \cdot\left(\delta^{p+1}+L^{p+1}\|z-\bar{x}\|^{p+1}\right)}{\sigma^{p}\cdot (p+1)!}
    + \frac{\sigma}{4\cdot (p+1)!}r^{p+1}.
\ea
\eeq
Now, combining (\ref{FuncReasoning}) and (\ref{FuncReasoning2}), it follows that
$$
\ba{rcl}
f(x^{+})
& \leq & 
f(\bar{x}) + \frac{(8(p + 1))^p  \cdot 
	\left(\delta^{p+1} + L^{p+1}\|z-\bar{x}\|^{p+1}\right)}{\sigma^{p}\cdot (p+1)!} 
	- \frac{\sigma}{4\cdot (p+1)!}r^{p+1},
\ea
$$
which implies that (\ref{FuncProg}) is true.
\end{proof}

Now, we can combine the two previous lemmas to obtain progress in the function value
in terms of the gradient norm.

\begin{lemma} \label{LemmaCombinedProg}
Suppose that Assumption A1 holds and let $x^{+}$ be an inexact minimizer of $M_{\bar{x},\sigma,p}(\,\cdot\,)$ defined in (\ref{ModelDef}), satisfying conditions (\ref{InexCond}) and (\ref{MonCond}). If $\sigma\geq 2L$ and, for some $z\in\mathbb{R}^{n}$ and $\delta>0$, (\ref{TCond}) holds, then
\small
\begin{equation} \label{FuncGradProg}
\ba{rcl}
f(\bar{x}) - f(x^{+})
&\geq& 
\frac{\|\nabla f(x^{+})\|^{\frac{p+1}{p}}}{2^5 \cdot 3^{1/p}\cdot \sigma^{1/p}\cdot (p+1)!}
+ \frac{\sigma}{8\cdot (p+1)!}\|x^{+}-\bar{x}\|^{p+1}
- \frac{c_p \cdot \left(\delta^{p+1}+L^{p+1}\|z-\bar{x}\|^{p+1}\right)}{\sigma^{p}\cdot (p+1)!},
\ea
\end{equation}
where $c_p :=  \frac{1}{2^5 \cdot p[ (p - 1)! ]^{p + 1}} + (8(p + 1))^p \leq 2 \cdot (8(p + 1))^p$.
\normalsize
\end{lemma}
\begin{proof}
Let $r = \|x^{+} - \bar{x}\|$. In view of Lemma~\ref{LemmaGradR}, we have
\begin{equation} \label{FuncGradReasoning}
\ba{rcl}
\frac{\sigma}{8\cdot (p+1)!}r^{p+1}
& \overset{(\ref{GradUpBound})}{\geq} &
\frac{\|\nabla f(x^{+})\|^{\frac{p+1}{p}}}{2^5 \cdot 3^{1/p}
	\cdot \sigma^{1/p}\cdot (p+1)!}
- \frac{\delta^{p+1}+L^{p+1}\|z-\bar{x}\|^{p+1}}{2^5 \cdot p [(p - 1)!]^{p + 1} \cdot \sigma^{p}\cdot (p+1)!}.
\ea
\end{equation}
Then, combining (\ref{FuncGradReasoning}) with inequality (\ref{FuncProg}) in Lemma~\ref{LemmaFuncProg}, we conclude that
$$
\ba{rcl}
    & & \!\!\!\!\!\!\!\!\!\!\!\!\!\!\!\!
    f(\bar{x})-f(x^{+})
    \overset{(\ref{FuncProg})}{\geq}
    \frac{\sigma}{8\cdot (p+1)!}r^{p+1}
    + \frac{\sigma}{8\cdot (p+1)!}r^{p+1}
    - \frac{(8(p + 1))^p \cdot \left(\delta^{p+1}+L^{p+1}\|z-\bar{x}\|^{p+1}\right)}{\sigma^{p}\cdot (p+1)!} \\
    \\
    &\overset{(\ref{FuncGradReasoning})}{\geq}& 
    \frac{\|\nabla f(x^{+})\|^{\frac{p+1}{p}}}{2^5 \cdot 3^{1/p}\cdot \sigma^{1/p}\cdot (p+1)!}
    + \frac{\sigma}{8\cdot (p+1)!}r^{p+1}
    - \frac{\delta^{p+1}+L^{p+1}\|z-\bar{x}\|^{p+1}}{\sigma^{p}\cdot (p+1)!} \cdot \Bigl[ 
    \frac{1}{2^5 \cdot p[(p - 1)!]^{p + 1}}
    + (8(p + 1))^p
    \Bigr],
\ea
$$
that is (\ref{FuncGradProg}) is true.
\end{proof}

Up to this moment, we have considered an arbitrary tensor $T \approx \nabla^p f(\bar{x})$
satisfying the $\delta$-approximation guarantee~(\ref{TCond}), for some $\delta > 0$. 
In this work, we are interested
in using for $T$ the finite-difference approximation provided by the $(p - 1)$th-order derivatives.
To this end, we use the following lemma.

\begin{lemma} \label{LemmaFiniteDiff}
	Suppose that Assumption A1 holds. Given $z \in \R^n$ and $h > 0$, let
	$T$ be the $p$-linear form defined by
	\beq \label{TDef}
	\ba{rcl}
	T	& = & 
	\sum\limits_{i = 1}^n \Bigl( \frac{\nabla^{p-1 }f(z + he_{i}) - \nabla^{p-1} f(z)}{h} \Bigr) 
		\otimes e_{i},
	\ea
	\eeq
where $e_i$ is the $i$th vector of the canonical basis of $\R^n$. Then
\beq \label{TApprox}
\ba{rcl}
\| T - \nabla^p f(z) \| & \leq & \frac{L \sqrt{n}}{2} h.
\ea
\eeq
\end{lemma}
\begin{proof}
Let us fix arbitrary directions $u_1, \ldots, u_{n - 1} \in \R^n$ s.t. $\|u_j\| \leq 1$ for all $j$.
Then, for an arbitrary $1 \leq i \leq n$, we have
\beq \label{TenEiBound}
\ba{rcl}
& & \!\!\!\!\!\!\!\!\!\!\!
| (T - \nabla^p f(z))[u_1, \ldots, u_{n - 1}, e_i] | \\
\\
& = &
\frac{1}{h} \Big| \bigl( \nabla^{p - 1} f(z + h e_i) - \nabla^{p - 1} f(z) - h \nabla^p f(z)[e_i]  \bigr) [u_1, \ldots, u_{n - 1}] \Big| \\
\\
& = &  \Big| 
\int\limits_0^1 (\nabla^p f(z + \tau h e_i) - \nabla^p f(z)) [u_1, \ldots, u_{n - 1}, e_i] d\tau
\Big|
\;\; \overset{(\ref{LipTensor})}{\leq} \;\;
L \int\limits_0^1 \tau h d\tau
\;\; = \;\; 
\frac{L}{2} h,
\ea
\eeq
where we used that $\nabla^p f$ is a symmetric form, and the standard Newton-Leibniz formula.
Therefore,
$$
\ba{rcl}
\|T - \nabla^p f(z) \|
& = & 
\max\limits_{\substack{u_1, \ldots, u_{n - 1}, x \in \R^n \\
\forall j \, \| u_j \| \leq 1, \|x\| \leq 1}}
| (T - \nabla^p f(z))[u_1, \ldots, u_{n - 1}, x] | \\
\\
& = & 
\max\limits_{\substack{u_1, \ldots, u_{n - 1}, x \in \R^n \\
		\forall j \, \| u_j \| \leq 1, \| x \| \leq 1 }}
| (T - \nabla^p f(z))[u_1, \ldots, u_{n - 1},  \sum\limits_{i = 1}^n x^{(i)} e_i  ] |  \\
\\
& \leq & 
\max\limits_{\substack{u_1, \ldots, u_{n - 1}, x \in \R^n \\
		\forall j \, \| u_j \| \leq 1, \| x \| \leq 1 }}
	\sum\limits_{i = 1}^n | x^{(i)} |
	\cdot
	| (T - \nabla^p f(z))[u_1, \ldots, u_{n - 1},  e_i  ] | \\
\\
& \overset{(\ref{TenEiBound})}{\leq} &
\max\limits_{x \in \R^n \, : \, \sum\limits_{i = 1}^n [ x^{(i)} ]^2 \leq 1}
\sum\limits_{i = 1}^n |x^{(i)} | \cdot \frac{L}{2} h
\;\; = \;\;
\frac{L\sqrt{n}}{2} h,
\ea
$$
where the last equation follows from the Cauchy-Schwartz inequality.
\end{proof}

\begin{remark}
	Note that the tensor $T$ defined by~\eqref{TDef} is not symmetric.
	Since our model~\eqref{ModelDef} depends only on the symmetric part of $T$,
	one can replace $T$ by $P_{sym}(T)$ (see Section~\ref{SectionNotation}),
	which possesses the same approximation guarantee:
	$$
	\ba{rcl}
	\| P_{sym}(T) - \nabla^p f(z) \| & \leq & \| T - \nabla^p f(z) \| 
	\;\; \overset{(\ref{TApprox})}{\leq} \;\;
	\frac{L \sqrt{n}}{2} h.
	\ea
	$$
	All our results remain valid regardless of whether $T$ is symmetric or not.
	However, the use of symmetrization affects the implementation of an inner solver for our subproblem 
	(see expression~(\ref{EeGrad}) for computing the gradient of $p$th-order term of our model).
\end{remark}

From~(\ref{TApprox}), we see that the key parameter that controls the finite-difference approximation error is the discretization step $h > 0$. If $h$ is chosen sufficiently small, we can ensure the same progress as in the method using the exact $p$th derivative.

\begin{lemma} \label{LemmaLazyFuncProg}
	Let $\epsilon > 0$ be fixed.
Suppose that Assumption A1 holds, and let $x^{+}$ be an inexact minimizer of the model $M_{\bar{x},\sigma,p}(\cdot)$ defined in (\ref{ModelDef}), where the tensor $T$ is constructed by finite differences as in Lemma~\ref{LemmaFiniteDiff}, with stepsize
\begin{equation} \label{HChoice}
\ba{rcl}
h &\leq&
\frac{4}{\sigma \sqrt{n}} \left[ \frac{\sigma^{p}\cdot\epsilon^{\frac{p+1}{p}}}{
	(8(p + 1))^p \cdot 2^7 \cdot 3^{1/p}\sigma^{1/p}} \right]^{\frac{1}{p+1}}.
\ea
\end{equation}
Assume further that $x^{+}$ satisfies conditions (\ref{InexCond}) and (\ref{MonCond}), and 
that $\|\nabla f(x^{+})\| \geq \epsilon$. If $\sigma \geq 2L$, then
\begin{equation} \label{FuncLazyProg}
\ba{rcl}
f(\bar{x}) - f(x^{+})
&\geq& 
\frac{\epsilon^{\frac{p+1}{p}}}{2^6 \cdot 3^{1/p} \cdot \sigma^{1/p}\cdot (p+1)!}
+ \frac{\sigma}{8\cdot (p+1)!}\|x^{+}-\bar{x}\|^{p+1}
- \frac{c_p L^{p+1}\|z-\bar{x}\|^{p+1}}{\sigma^{p}\cdot (p+1)!}.
\ea
\end{equation}
\end{lemma}
\begin{proof}
Since $\sigma\geq 2L$, it follows from (\ref{HChoice}) that
$$
\ba{rcl}
    h &\leq& 
    \frac{2}{L\sqrt{n}}\left[ \frac{\sigma^{p}
    	\cdot \epsilon^{\frac{p+1}{p}}}{ (8(p + 1))^p \cdot 2^7 \cdot 3^{1/p}\sigma^{1/p}} \right]^{\frac{1}{p+1}}.
\ea
$$
Thus, as $T$ is constructed by finite differences with stepsize $h$, Lemma~\ref{LemmaFiniteDiff} implies that
\beq \label{TDeltaBound}
\ba{rcl}
\|T-\nabla^{p}f(z)\|
& \leq &
\frac{L\sqrt{n}}{2}h
\;\; \leq \;\; \delta,
\ea
\eeq
where
\beq \label{DeltaDef}
\ba{rcl}
\delta & := & 
\Bigl[ \frac{\sigma^{p} \cdot \epsilon^{\frac{p+1}{p}}}{
	(8(p + 1))^p \cdot 2^7 \cdot 3^{1/p}\sigma^{1/p}} \Bigr]^{\frac{1}{p+1}}.
\ea
\eeq
Then, by (\ref{TDeltaBound}), Lemma~\ref{LemmaCombinedProg} and the assumption $\|\nabla f(x^{+})\|>\epsilon$, we have
\beq \label{FuncLazyProgReasoning}
\ba{rcl}
f(\bar{x})-f(x^{+})
&\geq&
\frac{\epsilon^{\frac{p+1}{p}}}{2^5 \cdot 3^{1/p} \cdot \sigma^{1/p}\cdot (p+1)!}
- \frac{ c_p \cdot \delta^{p+1}}{\sigma^{p}\cdot (p+1)!} \\
\\
& &
+ \frac{\sigma}{8\cdot (p+1)!}\|x^{+}-\bar{x}\|^{p+1} 
- \frac{ c_p \cdot L^{p+1}\|z-\bar{x}\|^{p+1}}{\sigma^{p}\cdot (p+1)!}.
\ea
\eeq
Note that by (\ref{DeltaDef}) and the bound $c_p \leq 2 \cdot (8(p + 1))^p$, we have
\beq \label{FuncLazyProgReasoning2}
\ba{rcl}
\frac{c_p \cdot\delta^{p+1}}{\sigma^{p}\cdot (p+1)!}
&=&
\frac{c_p}{2(8(p + 1))^p} \cdot
\frac{\epsilon^{\frac{p+1}{p}}}{2^6 \cdot 3^{1/p}\cdot \sigma^{1/p}\cdot (p+1)!}
\;\; \leq \;\;
\frac{\epsilon^{\frac{p+1}{p}}}{2^6 \cdot 3^{1/p}\cdot\sigma^{1/p}\cdot (p+1)!}.
\ea
\eeq
Then, combining (\ref{FuncLazyProgReasoning}) and (\ref{FuncLazyProgReasoning2}) we conclude that (\ref{FuncLazyProg}) is true.
\end{proof}

We are now ready to prove the main result of this section,
which serves as a building block of our algorithm.
We start with a fixed initialization $x_0 \in \R^n$ and perform $m \geq 1$
steps of the method using a fixed finite-difference approximation tensor $T$, computed 
once at $z := x_0$,
and a fixed regularization parameter $\sigma > 0$. 
We show that for an appropriate choice of $\sigma$, we guarantee strict progress
for the iterates of the method.

\begin{theorem} \label{TheoremLazyProgress}
Let $\epsilon > 0$ be fixed.
Suppose that Assumption A1 holds. 
Given $z\in\mathbb{R}^{n}$, $\sigma>0$, and $m\in\mathbb{N}\setminus\left\{0\right\}$, 
let $\left\{x_{t}\right\}_{t=0}^{m}$ be a sequence of points defined as follows
\small
\beq \label{ThProgIters}
\left\{\begin{array}{lll} x_{0}&=&z,\\
x_{t+1}&\in&\left\{y\in\mathbb{R}^{n}\,:\,M_{x_{t},\sigma,p}(y)\leq f(x_{t})\,\,\text{and}\,\,\|\nabla M_{x_{t},\sigma,p}(y)\|\leq\frac{\sigma}{2 p!}\|y-x_{t}\|^{p}\right\},\quad t=0,\ldots,m-1,
\end{array}
\right.
\eeq
\normalsize
where, for every $t\in\left\{0,\ldots,m-1\right\}$, the model is given by
\beq \label{IterModelDef}
\ba{rcl}
M_{x_{t},\sigma,p}(y) & \equiv & 
f(x_{t}) + \sum\limits_{i=1}^{p-1}\frac{1}{i!}\nabla^{i}f(x_{t})[y-x_{t}]^{i}
+ \frac{1}{p!}T[y-x_{t}]^{p}
+ \frac{\sigma}{(p+1)!}\|y-x_{t}\|^{p+1},
\ea
\eeq
with $T$ being a fixed tensor defined by finite differences, as in Lemma~\ref{LemmaFiniteDiff}, with 
stepsize $h$ satisfying~(\ref{HChoice}). If 
\beq \label{SigmaLarge}
\ba{rcl}
\sigma & \geq & 11 (p + 1)Lm,
\ea
\eeq
and
\beq \label{GradLarge}
\ba{rcl}
\|\nabla f(x_{i+1})\| & \geq &\epsilon\quad\text{for}\,\,i=0,\ldots,t,
\ea
\eeq
for some $t\in\left\{0,\ldots,m-1\right\}$, then
\beq \label{FuncEpsProgress}
\ba{rcl}
f(x_{0})-f(x_{t+1})
&\geq&
\frac{\epsilon^{\frac{p+1}{p}}}{2^6 \cdot 3^{1/p}\cdot\sigma^{1/p}\cdot (p+1)!}(t+1).
\ea
\eeq
\end{theorem}
\begin{proof}
Assume that $t \geq 1$ and consider $i \in \{0, \ldots, t\}$. 
In view of (\ref{ThProgIters})--(\ref{GradLarge}) and the construction of the tensor $T$, 
Lemma~\ref{LemmaLazyFuncProg} applies with $x^{+}=x_{i+1}$ and $\bar{x}=x_{i}$. 
Hence, we obtain
\beq \label{FuncEpsReasoning}
\ba{rcl}
f(x_{i})-f(x_{i+1})
&\geq&
\frac{\epsilon^{\frac{p+1}{p}}}{2^6 \cdot 3^{1/p}\cdot\sigma^{1/p}\cdot (p+1)!}
+ \frac{\sigma}{8\cdot (p+1)!}\|x_{i+1}-x_{i}\|^{p+1}
- \frac{ c_p \cdot L^{p+1}\|z - x_{i}\|^{p+1}}{\sigma^{p}\cdot (p+1)!}.
\ea
\eeq
Let us denote $r_{i} = \|x_{i+1} - x_{i}\|$. 
Summing inequalities (\ref{FuncEpsReasoning}) over $i = 0,\ldots,t$ and using $z = x_{0}$, we obtain
\beq \label{FuncEpsSummed}
\ba{rcl}
f(x_{0})-f(x_{t+1})
&\geq&
\frac{\epsilon^{\frac{p+1}{p}}}{2^6 \cdot 3^{1/p}\cdot\sigma^{1/p}\cdot (p+1)!}(t+1)
+ \frac{\sigma}{8\cdot (p+1)!}\sum\limits_{i=0}^{t}r_{i}^{p+1} \\
\\
& &
- \frac{ c_p \cdot L^{p+1}}{\sigma^{p}\cdot (p+1)!}
\sum\limits_{i=0}^{t}\|x_{i}-x_{0}\|^{p+1}.
\ea
\eeq
Note that
\beq \label{DistBound}
\ba{rcl}
    \sum\limits_{i=0}^{t}\|x_{i}-x_{0}\|^{p+1}
    &=&
    \sum\limits_{i=1}^{t}\|x_{i}-x_{0}\|^{p+1}
    \;\; = \;\;\sum\limits_{i=1}^{t}\left\|\sum\limits_{j=0}^{i-1}\,x_{j+1}-x_{j}\right\|^{p+1} \\
    \\
    &\leq& 
    \sum\limits_{i=1}^{t}\left(\sum\limits_{j=1}^{i-1}\|x_{j+1}-x_{j}\|\right)^{p+1} 
    \;\; = \;\;
    \sum\limits_{i=1}^{t}\left(\sum\limits_{j=0}^{i-1}r_{j}\right)^{p+1}.
\ea
\eeq
In addition, by the H\"{o}lder inequality, we also have
$$
\ba{rcl}
\sum\limits_{j=0}^{i-1} r_{j}
&\leq&
\left(\sum\limits_{j=0}^{i-1} 1^{\frac{p+1}{p}}\right)^{\frac{p}{p+1}}
\left(\sum\limits_{j=0}^{i-1}r_{j}^{p+1}\right)^{\frac{1}{p+1}}
\;\; \leq \;\; i^{\frac{p}{p+1}}\left(\sum\limits_{j=0}^{i-1}r_{j}^{p+1}\right)^{\frac{1}{p+1}}
\ea
$$
and so, since $i\leq t$, it follows that
\beq \label{HolderIneq}
\ba{rcl}
\left(\sum\limits_{j=0}^{i-1}r_{j}\right)^{p+1}
& \leq & 
i^{p}\left(\sum\limits_{j=0}^{i-1}r_{j}^{p+1}\right)
\;\; \leq \;\; t^{p}\left(\sum\limits_{j=0}^{t}r_{j}^{p+1}\right).
\ea
\eeq
Thus, combining (\ref{DistBound}) and (\ref{HolderIneq}) we get
\beq \label{DistBound2}
\ba{rcl}
\sum\limits_{i=0}^{t}\|x_{i}-x_{0}\|^{p+1}
& \leq & 
t^{p+1}\left(\sum\limits_{i=0}^{t}r_{i}^{p+1}\right).
\ea
\eeq
It follows from (\ref{FuncEpsSummed}) and (\ref{DistBound2}) that
\beq \label{FuncTelescopedBound}
\ba{rcl}
f(x_{0})-f(x_{t+1})
&\geq&
\frac{\epsilon^{\frac{p+1}{p}}}{2^6 \cdot 3^{1/p}\cdot\sigma^{1/p}\cdot (p+1)!}(t+1)
+ \left[\frac{\sigma}{8\cdot (p+1)!} - \frac{ c_p \cdot L^{p+1}\cdot t^{p + 1}}{\sigma^{p}\cdot (p+1)!}\right]
\sum\limits_{i=0}^{t}r_{i}^{p+1}.
\ea
\eeq
By (\ref{SigmaLarge}) we have
$$
\ba{rcl}
\sigma & \geq & 11(p + 1) L m
\;\; \geq \;\; (8 c_p)^{\frac{1}{p+1}}Lt,
\ea
$$
which implies that
\beq \label{SigmaIneq}
\ba{rcl}
\left[\frac{\sigma}{8\cdot (p+1)!} - \frac{c_p \cdot L^{p+1}\cdot t^{p+1}}{\sigma^{p}\cdot (p+1)!}\right]
& \geq & 0.
\ea
\eeq
Thus, by combining (\ref{FuncTelescopedBound}) and (\ref{SigmaIneq}), we conclude that (\ref{FuncEpsProgress}) holds for $t \geq 1$. Next, observe that when $t = 0$, inequality (\ref{FuncEpsReasoning}) also holds for $i = t = 0$, 
and the last term on its right-hand side vanishes, since in this case $x_{i} = x_{0} = z$. Consequently, (\ref{FuncEpsProgress}) also holds for $t = 0$, which completes the proof.
\end{proof}

\newpage

%%%%%%%%%%%%%%%%%%%%%%%%%%%%%%%%%%%%%%%%%%%%%%%%%%
%%%%%%%%%%%%%%%%%%%%%%%%%%%%%%%%%%%%%%%%%%%%%%%%%%
\section{A Lazy Method of Order $(p-1)$}
\label{SectionMethod}
%%%%%%%%%%%%%%%%%%%%%%%%%%%%%%%%%%%%%%%%%%%%%%%%%%
%%%%%%%%%%%%%%%%%%%%%%%%%%%%%%%%%%%%%%%%%%%%%%%%%%

In view of Theorem~\ref{TheoremLazyProgress}, let us define the algorithm
\[
(z^{+},\alpha)=\texttt{LazyTensorSteps}\left(z,T,\sigma,m,\epsilon\right)
\]
that attempts to perform $m$ inexact $p$th-order steps, starting from $z$, using the same tensor $T$ and regularization parameter $\sigma$, and recomputing the derivatives of $f$ up to order $p-1$ at each step. The algorithm stops earlier whenever an $\epsilon$-approximate stationary point of $f$ is found, or when the functional decrease with respect to $f(z)$ does not satisfy condition (\ref{FuncEpsProgress}). If the $m$ steps are performed, or the algorithm stops due to the violation of (\ref{FuncEpsProgress}), the output $z^{+}$ is the point with smallest function value among those generated by the algorithm. Otherwise, if an $\epsilon$-approximate stationary point is found, that point is returned as $z^{+}$. The output $\alpha$ specifies the reason why the algorithm stopped, and thus characterizes the type of the output.

\begin{algorithm}[h!]
	\caption{$(z^{+},\alpha)=\texttt{LazyTensorSteps}(x, T, \sigma, m, \epsilon)$} \label{alg:HessianFree}
	\SetKwInOut{Input}{input}\SetKwInOut{Output}{output}
	\BlankLine
	\textbf{Step 0.} Set $x_0 := x$, $\tilde{x}_{0}:=x$ and $t := 0$.
	\\[0.15cm]
	{\bf Step 1.} 	If $t = m$ then stop and \textbf{return} $(\tilde{x}_t, \, \text{\ttfamily success})$.
	\\[0.15cm]
	{\bf Step 2.}  Compute $x_{t + 1}$ as an approximate solution to the subproblem
	$$
	\ba{rcl}
	\min\limits_{y\in\mathbb{R}^{n}}\,M_{x_{t},\sigma,p}(y)
	& \equiv & f(x_{t}) + \sum\limits_{i=1}^{p-1}\frac{1}{i!}\nabla^{i}f(x_{t})[y-x_{t}]^{i}
	+ \frac{1}{p!}T[y-x_{t}]^{p}+\frac{\sigma}{(p+1)!}\|y-x_{t}\|^{p+1},
	\ea
	$$
	such that
    $$
    \ba{rcl}
        M_{x_{t},\sigma,p}(x_{t+1}) & \leq & f(x_{t})
        \quad\text{and}\quad\|
        \nabla M_{x_{t},\sigma,p}(x_{t+1})\|
        \;\; \leq \;\; \frac{\sigma}{2 p!}\|x_{t+1}-x_{t}\|^{p}.
    \ea
    $$
    Define $\tilde{x}_{t+1}=\arg\min\left\{f(y)\,:\,y\in\left\{\tilde{x}_{t},x_{t+1}\right\}\right\}$.
    \\[0.15cm]
	{\bf Step 3.}
	If $\| \nabla f(x_{t + 1})\| \leq \epsilon$ then stop and \textbf{return} $(x_{t + 1}, \, \text{\ttfamily solution})$.
	\\[0.15cm]
	{\bf Step 4.} If 
    \beq \label{eq:sufficient_decrease}
    \ba{rcl}
    f(x_{0})-f(\tilde{x}_{t+1})
    & \geq &
    \frac{\epsilon^{\frac{p+1}{p}}}{2^6 \cdot 3^{1/p}\cdot\sigma^{1/p}\cdot (p+1)!}(t+1).
    \ea
    \eeq
	holds then set $t := t + 1$ and go to Step 1. Otherwise, stop and \textbf{return} $(\tilde{x}_{t + 1}, \text{\ttfamily halt})$.
\end{algorithm}

\noindent Algorithm 1 outputs both the point $z^{+}$ and the status indicator
\[
\alpha \in \{\texttt{success}, \, \texttt{solution}, \, \texttt{halt}\},
\]
which specifies the termination condition. The value \texttt{success} indicates that all prescribed $m$ steps were completed; \texttt{solution} signals that an $\epsilon$-approximate stationary point of $f$ was identified; and \texttt{halt} means that progress in reducing the objective function was insufficient. As a direct consequence of Theorem~\ref{TheoremLazyProgress}, we have the following result.

\begin{lemma}
\label{lem:2.7}
Suppose that Assumption A1 holds. Given $z\in\mathbb{R}^{n}$, $\epsilon>0$, $\sigma>0$ and $m\in\mathbb{N}\setminus\left\{0\right\}$, let $(z^{+},\alpha)$ be the corresponding output of Algorithm 1 with
\beq \label{eq:fd1}
\ba{rcl}
T &=& 
\sum\limits_{i=1}^{n}
\Bigl( \frac{\nabla^{p-1}f(z+he_{i})-\nabla^{p-1}f(z)}{h} \Bigr) \otimes e_{i},
\ea
\eeq
for some $h>0$. If 
\vspace{-0.5cm}
\beq \label{eq:fd2}
\ba{rcl}
    \sigma & \geq & 11(p + 1)Lm
    \quad\text{and}\quad 
    h \;\; \leq \;\; \frac{4}{\sigma \sqrt{n}} 
    \left[ \frac{\sigma^{p}\cdot\epsilon^{\frac{p+1}{p}}}{ (8(p + 1))^{p} \cdot 2^7 \cdot 3^{1/p}\sigma^{1/p}} \right]^{\tfrac{1}{p+1}}.
\ea
\eeq
then either $\alpha=\texttt{solution}$ (and so $\|\nabla f(z^{+})\|\leq\epsilon$) or $\alpha=\texttt{success}$ and so 
\beq \label{eq:good_decrease}
\ba{rcl}
    f(z) - f(z^{+})
    & \geq &
    \frac{\epsilon^{\frac{p+1}{p}}}{2^6 \cdot 3^{1/p}\cdot \sigma^{1/p}\cdot (p+1)!}m.
\ea
\eeq
\end{lemma}

\begin{proof}
Suppose that $\alpha\neq\texttt{solution}$. Then $z^{+}=\tilde{x}_{t_{*}+1}$ for some $t_{*}<m$ and 
\beq \label{eq:non_criticality}
\ba{rcl}
\|\nabla f(x_{i+1})\|
& \geq &\epsilon\quad\text{for}\,\,i=0,\ldots,t_{*}.
\ea
\eeq
In view of (\ref{eq:fd1}), (\ref{eq:fd2}), (\ref{eq:non_criticality}), and Theorem~\ref{TheoremLazyProgress}, it follows that
\beq     \label{eq:intermediate_decrease}
\ba{rcl}
     f(x_{0})-f(x_{t_{*}+1})
     & \geq &
     \frac{\epsilon^{\frac{p+1}{p}}}{2^6 \cdot 3^{1/p}\cdot\sigma^{1/p}(p+1)!}(t_{*}+1).
\ea
\eeq
Since $f(\tilde{x}_{t_{*}+1})\leq f(x_{t_{*}+1})$, this means that condition (\ref{eq:sufficient_decrease}) was satisfied for $t=t_{*}$. Therefore, the only way Algorithm 1 could have returned $z^{+} = \tilde{x}_{t_{*}+1}$ is if $t_{*} + 1 = m$. Consequently, $\alpha = \texttt{success}$, and by (\ref{eq:intermediate_decrease}) we conclude that (\ref{eq:good_decrease}) holds.
\end{proof}

Building upon Algorithm 1, we can design an adaptive lazy scheme that emulates a $p$th-order method while using only a lower-order oracle of order $p-1$. At the $k$th iteration of the lazy method, we have the current iterate $z_{k}$ and an estimate $L_{k}$ of the Lipschitz constant $L$. Defining 
$$
\ba{rcl}
\sigma_{k} & = & 11(p + 1) L_{k} m 
\quad \text{and} \quad 
h_{k} \;\; = \;\; 
\frac{4}{\sigma_{k} \sqrt{n}} 
\left[ \frac{\sigma_{k}^{p}\cdot\epsilon^{\frac{p+1}{p}}}{ (8(p + 1))^p \cdot 2^7 \cdot 3^{1/p}\cdot\sigma_{k}^{1/p}} \right]^{\frac{1}{p+1}},
\ea
$$
we then compute a $p$th-order tensor $T_{k}$ as in Lemma \ref{lem:2.7} (with $z = z_{k}$ and $h=h_{k}$). The next iterate $z_{k+1}$ is computed by calling Algorithm 1: 
$$
\ba{rcl}
(z_{k+1}, \alpha_{k}) & = & 
\texttt{LazyTensorSteps}(z_{k}, T_{k}, \sigma_{k}, m, \epsilon).
\ea
$$  
If $\alpha_{k} = \texttt{solution}$, this means that $\|\nabla f(z_{k+1})\| \leq \epsilon$, and the algorithm terminates. If $\alpha_{k}=\texttt{success}$, then it follows from Lemma \ref{lem:2.7} that
$$
\ba{rcl}
    f(z_{k})-f(z_{k+1})
    & \geq &
    \frac{\epsilon^{\frac{p+1}{p}}}{2^6 \cdot 3^{1/p} \cdot \sigma_{k}^{1/p}\cdot (p+1)!}m.
\ea
$$
In this case, we say that the iteration was \emph{successful} and, to allow a larger step in the next iteration, we set \(L_{k+1} = L_{k}/2\). Conversely, if \(\alpha_{k} = \texttt{halt}\), we say that the \(k\)th iteration was \emph{unsuccessful} and set \(L_{k+1} = 2 L_{k}\). In what follows we provide  a detailed description of this algorithm.

%\newpage
\begin{algorithm}[h!]
	\caption{\textbf{Lower-Order Implementation of a $p$th-Order Method}} \label{alg:FirstOrderCNM}
    \BlankLine
	\noindent\textbf{Step 0.} Given $z_{0}\in\mathbb{R}^{n}$, $L_{0}>0$, $\epsilon>0$, and $m\in\mathbb{N}\setminus\left\{0\right\}$, set $k:=0$.
    \\[0.2cm]
    \noindent\textbf{Step 1.} If $\|\nabla f(z_{k})\|\leq\epsilon$, stop.
    \\[0.2cm]
    \noindent\textbf{Step 2.} Using
    \begin{equation} \label{eq:3.1}
    \ba{rcl}
    \sigma_{k} & = & 11(p + 1) L_{k} m 
    \quad \text{and} \quad 
    h_{k} \;\; = \;\; \frac{4}{\sigma_{k} \sqrt{n}} 
    \Bigl[ \frac{\sigma_{k}^{p}\cdot\epsilon^{\frac{p+1}{p}}}{ (8(p + 1))^p \cdot 2^7 \cdot 3^{1/p}\sigma_{k}^{1/p}} \Bigr]^{\frac{1}{p+1}},
    \ea
    \end{equation}
    compute the finite difference tensor:
    \begin{equation} \label{eq:3.2}
    \ba{rcl}
    T_{k}
    &= &
    \sum\limits_{i=1}^{n}
    \Bigl( \frac{\nabla^{p-1}f(z_{k}+h_{k}e_{i})-\nabla^{p-1}f(z_{k})}{h_{k}} \Bigr)
    \otimes e_{i}.
    \ea
    \end{equation}
    \noindent\textbf{Step 3.} Attempt to perform $m$ lazy tensor steps using the same tensor $T_{k}$:
    \begin{equation}
    (z_{k+1}, \alpha_{k}) = \texttt{LazyTensorSteps}(z_{k}, T_{k}, \sigma_{k}, m, \epsilon).
    \label{eq:3.3}
    \end{equation}
    \noindent\textbf{Step 4.} Update the estimate of the Lipschitz constant:
    \begin{equation}
    L_{k+1}=\left\{\begin{array}{ll}2L_{k}&\text{if $\alpha_{k}=\texttt{halt}$},\\
    L_{k}&\text{if $\alpha_{k}=\texttt{solution}$},\\
    L_{k}/2&\text{if $\alpha_{k}=\texttt{success}$.}
    \end{array}
    \right.
    \label{eq:3.4}
    \end{equation}
    \noindent\textbf{Step 5.} Set $k:=k+1$ and go back to Step 1.
\end{algorithm}

Note that the number of consecutive unsuccessful iterations is finite: by doubling \(L_{k}\), we will eventually have \(\sigma_{k} = 11 (p + 1) L_{k} m \geq 11(p + 1) L m\), and by Lemma~\ref{lem:2.7} this ensures that \(\alpha_{k} \in \left\{\texttt{solution}, \texttt{success}\right\}\). This fact allows us to establish the following upper bound for the sequence of estimates $L_{k}$.

\newpage
\begin{lemma}
\label{lem:3.2}
Suppose that Assumption A1 holds and let $\left\{z_{k}\right\}_{k=0}^{K}$ be generated by Algorithm 2. Then
\beq \label{eq:3.5}
\ba{rcl}
L_{k} & \leq & 
L_{\max} \;\; \equiv \;\;
\max\left\{L_{0},2L\right\}
\ea
\eeq
for all $k\in\left\{0,\ldots, K\right\}$.
\end{lemma}

\begin{proof}
Let us show it by induction over $k$. By the definition of $L_{\max}$, it follows that (\ref{eq:3.5}) is true for $k=0$. Suppose $K\geq 1$ and that (\ref{eq:3.5}) holds for some $k\in\left\{0,\ldots,K-1\right\}$. 
\\[0.2cm]
\noindent\textbf{Case I:} $\alpha_{k}\in\left\{\texttt{solution},\texttt{success}\right\}$.
\\[0.2cm]
In this case, by (\ref{eq:3.4}) and the induction assumption we have
$$
\ba{rcl}
    L_{k+1} & \leq & L_{k} \;\; \leq \;\; L_{\max},
\ea
$$
that is, (\ref{eq:3.5}) holds for $k+1$.
\\[0.2cm]
\noindent\textbf{Case II:} $\alpha_{k}=\texttt{halt}$.
\\[0.2cm]
In this case, we must have
\beq \label{eq:3.6}
\ba{rcl}
L_{k} & < & L,
\ea
\eeq
since otherwise we would obtain $\sigma_{k} = 11(p + 1)L_{k}m \geq 11(p + 1)Lm$. By Lemma~\ref{lem:2.7}, this would imply $\alpha_{k} \in \left\{\texttt{solution}, \texttt{success}\right\}$, contradicting the hypothesis of the present case. Hence, in view of (\ref{eq:3.4}), we deduce
$$
\ba{rcl}
L_{k+1} & = & 2L_{k} 
\;\; \overset{(\ref{eq:3.6})}{<} \;\; 2L \;\; \leq \;\; L_{\max},
\ea
$$
which shows that (\ref{eq:3.5}) also holds for $k+1$.
\end{proof}

Given $\left\{z_{k}\right\}_{k=0}^{K}$ generated by Algorithm~\ref{alg:FirstOrderCNM} with $K\geq 1$, let
\begin{eqnarray}
\mathcal{S}_{K}&=&\left\{k\in\left\{0,\ldots,K-1\right\}\,:\,\alpha_{k}=\texttt{success}\right\}\label{eq:3.7},\\
\mathcal{U}_{K}&=&\left\{k\in\left\{0,\ldots,K-1\right\}\,:\,\alpha_{k}=\texttt{halt}\right\}.\label{eq:3.8}
\end{eqnarray}

\begin{lemma} \label{lem:3.3}
Suppose that Assumptions A1-A2 hold and let $\left\{z_{k}\right\}_{k=0}^{K}$ be generated by Algorithm~\ref{alg:FirstOrderCNM} with $K\geq 2$. Then
\beq \label{eq:3.9}
\ba{rcl}
|\mathcal{S}_{K-1}|
& \leq  & 
\frac{2^6 \cdot\left(3\cdot 11(p + 1) \cdot L_{\max}\right)^{1/p}(p+1)!(f(z_{0})-f_{low})}{m^{(p-1)/p}}
\cdot 
\epsilon^{-\frac{p+1}{p}}.
\ea
\eeq
\end{lemma}

\begin{proof}
Since $z_{K}$ has been generated, it follows that
$$
\ba{rcl}
    \|\nabla f(z_{k+1})\| & \geq & \epsilon,\quad k\in\left\{0,\ldots,K-2\right\}.
\ea
$$
By Lemma~\ref{lem:2.7}, we then have
\begin{equation*}
 \alpha_{k}\in\left\{\texttt{success},\texttt{halt}\right\},\quad k\in\left\{0,\ldots,K-2\right\}.
\end{equation*}
Consequently, from (\ref{eq:3.3}) together with Steps 1 and 4 of Algorithm~\ref{alg:HessianFree}, we have
\beq \label{eq:3.10}
\ba{rcl}
f(z_{k+1})
&\leq& 
f(z_{k}),\quad k\in\left\{0,\ldots,K-2\right\}.
\ea
\eeq
Finally, combining A1, (\ref{eq:3.10}), (\ref{eq:3.7}) and Lemma \ref{lem:3.2} we deduce
$$
\ba{rcl}
    f(z_{0})-f_{low}&\geq &
    f(z_{0})-f(z_{T-1})=\sum\limits_{k=0}^{K-2}f(z_{k})-f(z_{k+1})\\
    \\
    &\geq &
    \sum\limits_{k\in\mathcal{S}_{K-1}} f(z_{k})-f(z_{k+1})\\
    \\
    &\geq &
    \sum\limits_{k\in\mathcal{S}_{K-1}} \frac{\epsilon^{\frac{p+1}{p}}}{2^6 \cdot 3^{1/p}\cdot\sigma_{k}^{1/p}\cdot (p+1)!}m\\
    \\
    & = &
    \sum\limits_{k\in\mathcal{S}_{K-1}} \frac{\epsilon^{\frac{p+1}{p}}}{2^6 \cdot 3^{1/p} \cdot (11(p + 1) L_{k}m)^{1/p}\cdot (p+1)!}m\\
    \\
    &\geq &
    \sum\limits_{k\in\mathcal{S}_{K-1}} \frac{\epsilon^{\frac{p+1}{p}}}{2^6 \cdot (3\cdot 11(p + 1) L_{\max}m)^{1/p}\cdot (p+1)!}m\\
    \\
    &=& 
    \frac{\epsilon^{\frac{p+1}{p}}}{2^6 \cdot (3\cdot 11(p + 1) L_{\max})^{1/p}\cdot (p+1)!}m^{(p-1)/p}|\mathcal{S}_{K-1}|,
\ea
$$
and therefore,
$$
\ba{rcl}
    |\mathcal{S}_{K-1}| &\leq& 
    \frac{2^6 \cdot\left(3\cdot 11(p + 1) \cdot L_{\max}\right)^{1/p}(p+1)!(f(z_{0})-f_{low})}{m^{(p-1)/p}}\epsilon^{-\frac{p+1}{p}}.
\ea
$$
\end{proof}

\begin{lemma}
\label{lem:3.4}
Suppose that Assumptions A1-A2 hold and let $\left\{z_{k}\right\}_{k=0}^{K}$ be generated by Algorithm~\ref{alg:FirstOrderCNM} with $K\geq 2$. Then
\beq \label{eq:3.11}
\ba{rcl}
|\mathcal{U}_{K-1}|
& \leq & 
\frac{2^6\cdot\left(3\cdot 11(p + 1) \cdot L_{\max}\right)^{1/p}(p+1)!(f(z_{0})-f_{low})}{m^{(p-1)/p}}\epsilon^{-\frac{p+1}{p}}
+ \log_{2}\left(\frac{L_{\max}}{L_{0}}\right).
\ea
\eeq
\end{lemma}

\begin{proof}
As in the proof of Lemma \ref{lem:3.3}, the fact that $z_{K}$ has been generated implies  
\[
\alpha_{k} \in \{\texttt{success}, \texttt{halt}\},\quad k\in\left\{0,\ldots,K-2\right\}.
\]  
Therefore, by (\ref{eq:3.4}) and Lemma \ref{lem:3.2}, we obtain  
$$ 
\ba{rcl}
L_{0}\left(\frac{1}{2}\right)^{|\mathcal{S}_{K-1}|}\cdot 2^{|\mathcal{U}_{K-1}|} 
& = & L_{K-1} 
\;\; \leq \;\; L_{\max}.
\ea
$$
Dividing both sides by $L_{0}$ and then taking the logarithm, it follows that
$$
\ba{rcl}
-|\mathcal{S}_{K-1}| + |\mathcal{U}_{K-1}| 
& \leq & \log_{2}\!\left(\frac{L_{\max}}{L_{0}}\right),
\ea
$$
and hence  
$$
\ba{rcl}
|\mathcal{U}_{K-1}| & \leq & 
|\mathcal{S}_{K-1}| + \log_{2}\!\left(\frac{L_{\max}}{L_{0}}\right).
\ea
$$
Thus, by Lemma \ref{lem:3.3}, relation (\ref{eq:3.11}) follows.
\end{proof}

Now we can establish an iteration-complexity bound of $\mathcal{O}\left(L_{\max}^{1/p}(f(z_{0})-f_{low})\epsilon^{-\frac{p+1}{p}}\right)$ for Algorithm~\ref{alg:FirstOrderCNM}.

\begin{theorem}
\label{thm:3.1}
Let $\epsilon > 0$ be fixed.
Suppose that Assumptions A1–A2 hold, and let 
$K(\varepsilon)\in\mathbb{N}\cup\{+\infty\}$ denote the termination time 
of Algorithm~\ref{alg:FirstOrderCNM}, i.e., the smallest index such that 
$\|\nabla f(x_{K(\varepsilon)})\|\le\varepsilon$. 
Then
\beq \label{eq:3.12}
\ba{rcl}
    K(\epsilon)& \leq & 
    1+ \frac{2^7\cdot\left(3\cdot 11(p + 1)\cdot L_{\max}\right)^{1/p}(p+1)!(f(z_{0})-f_{low})}{m^{(p-1)/p}}\epsilon^{-\frac{p+1}{p}}
    + \log_{2}\left(\frac{L_{\max}}{L_{0}}\right).
\ea
\eeq
\end{theorem}
\begin{proof}
If $K(\varepsilon)\leq 1$, then (\ref{eq:3.12}) follows immediately. 
Hence, suppose that $K(\varepsilon)\geq 2$. 
In this case, Lemmas~\ref{lem:3.3} and~\ref{lem:3.4} imply that
$$
\ba{rcl}
K(\varepsilon)-1
& = & 
|\mathcal{S}_{K(\varepsilon)-1}|+|\mathcal{U}_{K(\varepsilon)-1}| \\
\\
& \leq  &
\dfrac{2^7\,(3\cdot 11(p + 1) \, L_{\max})^{1/p}(p+1)!\,(f(z_{0})-f_{\mathrm{low}})}
{m^{(p-1)/p}}\,\varepsilon^{-\frac{p+1}{p}}
+ \log_{2}\!\left(\frac{L_{\max}}{L_{0}}\right),
\ea
$$
and therefore (\ref{eq:3.12}) holds.
\end{proof}

For $x\in\mathbb{R}^{n}$, a single call to the oracle of order $(p-1)$ corresponds to the computation of any nonempty subset of 
\[
\{f(x), \nabla f(x), \ldots, \nabla^{p-1} f(x)\}.
\]
At iteration $t$, Algorithm~\ref{alg:HessianFree} computes
\[
\{f(x_t), \nabla f(x_t), \ldots, \nabla^{p-1} f(x_t)\} \quad \text{and} \quad
\{f(x_{t+1}), \nabla f(x_{t+1})\},
\] 
which amounts to two oracle calls. Since Algorithm~\ref{alg:HessianFree} performs at most $m$ iterations, each run requires at most $2m$ oracle calls. In turn, at iteration $k$, Algorithm~\ref{alg:FirstOrderCNM} executes Algorithm~\ref{alg:HessianFree}, computes $\nabla f(z_k)$ at Step~1, and computes 
\(\{\nabla^{p-1} f(z_k + h_k e_i)\}_{i=1}^n\) at Step~2. 
Therefore, each iteration of Algorithm~\ref{alg:FirstOrderCNM} requires at most $2m + (1+n)=\mathcal{O}\left(m+n\right)$ oracle calls. Thus, by Theorem~\ref{thm:3.1}, Algorithm~\ref{alg:FirstOrderCNM} requires at most 
\beq \label{eq:3.13}
\ba{c}
    \mathcal{O}\left(\left[\frac{m + n}{m^{(p-1)/p}}\right]L_{\max}^{1/p}(f(z_{0})-f_{low})\epsilon^{-\frac{p+1}{p}}\right)
\ea
\eeq
calls to the oracle of order $(p-1)$ to find an $\varepsilon$-approximate stationary point of $f$. Recall that $m\in\mathbb{N}\setminus\{0\}$ is a user-defined parameter specifying the maximum number of iterations of Algorithm~2 (i.e., the number of lazy tensor steps). A natural question then arises: \textit{how should one choose $m$?} One approach is to select $m$ so as to minimize the corresponding factor in the oracle complexity bound~(\ref{eq:3.13}). Let 
$$
\ba{rcl}
    \varphi(m) & = & \frac{m+n}{m^{(p-1)/p}}.
\ea
$$
For $p = 1$, the optimal value is $m = 1$.
Consider $p \geq 2$.
Note that the minimum of the unimodal function $\varphi(\,\cdot\,)$ is achieved when 
$$
\ba{rcl}
    0 & = & \varphi'(m)
    \;\; = \;\; m^{-\frac{2p-1}{p}}\left[m-\frac{(p-1)}{p}(m+n)\right]
\ea
$$
which holds if, and only if,
$
m = (p-1)n.
$
Therefore, with such a choice of $m$ we obtain the following oracle complexity bound for 
Algorithm~\ref{alg:FirstOrderCNM}, combining both cases $p = 1$ and $p \geq 2$.

\begin{corollary}
\label{cor:3.1}
Suppose that A1-A2 holds. Then Algorithm~\ref{alg:FirstOrderCNM} with \fbox{$m=(p-1)n + 1$} requires at most 
\beq \label{Complexity}
   \mathcal{O}\left(n^{1/p}L_{\max}^{1/p}(f(z_{0})-f_{low})\epsilon^{-\frac{p+1}{p}}\right)
\eeq
calls to the oracle of order $(p-1)$ to find an $\epsilon$-approximate stationary point of $f(\,\cdot\,)$.
\end{corollary}

In particular, our analysis shows that for $p = 1$ (the zeroth-order approximation of the first-order method), 
the best strategy is $m = 1$, which means not using stale derivatives and updating the information at every iteration.

For $p \geq 2$, we show that the best strategy is to choose $m \propto n$ (i.e., updating the $p$th-order tensor approximation once every $(p - 1)n$ iterations). It is also interesting that, as the order of the method $p$ increases, we obtain an improved dependence on the problem dimension in the oracle complexity~(\ref{Complexity}).

%%%%%%%%%%%%%%%%%%%%%%%%%%%%%%%%%%%%%%%%%%%%%%%%%%
%%%%%%%%%%%%%%%%%%%%%%%%%%%%%%%%%%%%%%%%%%%%%%%%%%
\section{Discussion}
\label{SectionDiscussion}
%%%%%%%%%%%%%%%%%%%%%%%%%%%%%%%%%%%%%%%%%%%%%%%%%%
%%%%%%%%%%%%%%%%%%%%%%%%%%%%%%%%%%%%%%%%%%%%%%%%%%

In summary, for objectives with Lipschitz continuous $p$th-order derivatives, it is known \cite{Carmon} that the optimal oracle complexity of $p$th-order tensor methods for finding an $\epsilon$-approximate stationary point is $\mathcal{O}\!\left(\epsilon^{-\frac{p+1}{p}}\right)$, a bound attained by regularized $p$th-order methods (e.g., \cite{Birgin}). Our contribution is a method of order $(p-1)$ which achieves complexity $\mathcal{O}\!\left(n^{1/p}\epsilon^{-\frac{p+1}{p}}\right)$, thereby extending the result of \cite{DoikovGrapiglia} beyond the case $p=2$. To complete the picture, complexity bounds of order 
\(\mathcal{O}(n\epsilon^{-2})\) and \(\mathcal{O}(n^{3/2}\epsilon^{-3/2})\) 
have been established in \cite{grapiglia1} and \cite{DoikovGrapiglia}, 
respectively, for zeroth-order methods applied to the classes of functions 
with Lipschitz continuous gradient and Lipschitz continuous Hessian.
These results are summarized in Table~\ref{tab:1}, where the present work fills the entries highlighted in gray.

\setlength{\tabcolsep}{6pt} % default is 6pt
\renewcommand{\arraystretch}{1.5} % default is 1.0
\begin{table}[htp!]
\centering
\begin{tabular}{|c|c|c|c|c|c|c|}
    \hline
    \rowcolor{tableheader}
    \color{white}\textbf{Method/Class} & \color{white}$\nabla f$ Lip. & \color{white}$\nabla^{2}f$ Lip. & \color{white}$\nabla^{3}f$ Lip. & \,\,\,\color{white}$\ldots$\,\,\,& \color{white} $\nabla^{p-1}f$ Lip. & \color{white} $\nabla^{p}f$ Lip. \\
    \hline
    \textbf{zeroth-order} & $\mathcal{O}\left(n\epsilon^{-2}\right)$ 
    & $\mathcal{O}\left(n^{\frac{3}{2}}\epsilon^{-\frac{3}{2}}\right)$ 
    & ? &  & ?  & ?\\
    \hline
    \textbf{1st-order} & $\mathcal{O}\left(\epsilon^{-2}\right)$ 
    & $\mathcal{O}\left(n^{\frac{1}{2}}\epsilon^{-\frac{3}{2}}\right)$  
    & ?  &  & ? & ? \\
    \hline
    \textbf{2nd-order} & --  
    & $\mathcal{O}\left(\epsilon^{-\frac{3}{2}}\right)$ 
    & \cellcolor{gray!20}$\mathcal{O}\left(n^{\frac{1}{3}}\epsilon^{-\frac{4}{3}}\right)$ 
    &  & ? & ?\\
    \hline
    \textbf{3rd-order} & --  & -- 
    & $\mathcal{O}\left(\epsilon^{-\frac{4}{3}}\right)$ 
    & \cellcolor{gray!20}$\shallowddots$ 
    & ? & ? \\
    \hline
    $\vdots$ & --  & -- 
    & -- & $\shallowddots$ 
    & \cellcolor{gray!20}$\mathcal{O}\left(n^{\frac{1}{p-1}}\epsilon^{-\frac{p}{p-1}}\right)$  & ? \\
     \hline
    \textbf{$(p-1)$th-order} & --  & -- 
    & -- & 
    & $\mathcal{O}\left(\epsilon^{-\frac{p}{p-1}}\right)$& \cellcolor{gray!20}$\mathcal{O}\left(n^{\frac{1}{p}}\epsilon^{-\frac{p+1}{p}}\right)$ \\
    \hline
    \textbf{$p$th-order} & --  & -- 
    & --  &  & -- & $\mathcal{O}\left(\epsilon^{-\frac{p+1}{p}}\right)$ \\
    \hline
\end{tabular}
\caption{Summary of the best-known complexity bounds for finding $\epsilon$-approximate stationary points using $q$th-order methods applied to functions with Lipschitz continuous $p$th-order derivatives ($p \geq q$).}
\label{tab:1}
\end{table}

\newpage

In view of Table~\ref{tab:1}, to complete our understanding of the worst-case complexity of  $q$th-order method applied to functions with Lipschitz continuous $p$th-order derivatives ($p \geq q$), the case 
\begin{equation*}
\fbox{$q < p-1,\quad\text{for}\quad p \geq 3$}
\end{equation*}
still needs to be investigated. Even for the known bounds, outside the main diagonal ($q=p$) it is unclear whether they can be improved, as the question of their tightness ramains open. From a practical perspective, it would be interesting to design a less lazy variant of Algorithm~1 that incorporates quasi-tensor updates~\cite{Welzel}, as in~\cite{CartisJerad}, while still preserving the improved complexity bound with respect to the problem dimension $n$. Additionally, exploring universal lower-order implementations of higher-order methods for convex and nonconvex functions with H\"{o}lder-continuous derivatives would also be of interest \cite{GrapigliaNesterov,Martinez,CGT2}. We plan to address these questions in future work.

\section*{Acknowledgements}

The authors are very grateful to Coralia Cartis and Sadok Jerad for the interesting and motivating discussion, 
which resulted in an improved version of this paper.

\bibliographystyle{plain}
\bibliography{bibliography}

\end{document}